\documentclass[12pt,reqno,hypertex]{amsart}

\usepackage[english]{babel}
\usepackage{array,booktabs}

\usepackage{fullpage}

\usepackage{tikz}
\usetikzlibrary{automata}
\usetikzlibrary{positioning}
\usetikzlibrary{arrows,shapes}
\usetikzlibrary{decorations.markings}

\usepackage{graphicx}
\usepackage[alphabetic]{amsrefs}

\usepackage{float}
\restylefloat{table}

\newcommand{\C}{\mathbb{C}}
\newcommand{\Z}{\mathbb{Z}}
\newcommand{\scP}{\mathcal{P}}
\newcommand{\V}{V({\bf n})}
\newcommand{\K}{K({\bf n})}
\newcommand{\tI}{\tilde{I}_{2m}}

\theoremstyle{plain} 
\newtheorem*{thm*}{Theorem}
\newtheorem{thm}{Theorem}

\newtheorem{lemma}[thm]{Lemma}

\begin{document}

\title{Invariant polynomial functions on tensors under the action of a product of orthogonal groups}
\date{September 2012}

\author{Lauren Kelly Williams}


\maketitle

\begin{abstract}
Let $K$ be the product  $O({n_1})\times O({n_2}) \times \cdots \times O({n_r})$ of orthogonal groups. Let $V = \otimes_{i = 1}^r \C^{n_i}$, the $r$-fold tensor product of defining representations of each orthogonal factor. We compute a stable formula for the dimension of the $K$-invariant algebra of degree $d$ homogeneous polynomial functions on $V$. To accomplish this, we compute a formula for the number of matchings which commute with a fixed permutation. Finally, we provide formulas for the invariants and describe a bijection between a basis for the space of invariants and the isomorphism classes of certain $r$-regular graphs on $d$ vertices, as well as a method of associating each invariant to other combinatorial settings such as phylogenetic trees.
\end{abstract}

\section{Introduction}

Let $r \in \Z_{+}$, $d \in \Z_{\geq 0}$, and let ${\bf n} = (n_1,n_2,\ldots,n_r)$ be a sequence of positive integers. Let $\V = \C^{n_1} \otimes \C^{n_2} \otimes \cdots \otimes \C^{n_r}$ and $G({\bf n}) = GL_{n_1}(\C) \times GL_{n_2}(\C) \times \cdots \times GL_{n_r}(\C)$, where $GL_{n_i}(\C)$ denotes the general linear group of degree $n_i$ over the complex numbers. Finally, we denote by $K({\bf n})$ the product $O({n_1}) \times O({n_2})\times \cdots \times O({n_r})$, where $O({n_i})$ denotes the complex orthogonal group of degree $n_i$ defined by $O({n_i}) = \{g \in GL_{n_i}(\C) \ | \ g^Tg = I \}$. The group $\K$ acts on $\V$ by
\[ (g_1, g_2,\ldots,g_r) \cdot (v_v,v_2,\ldots,v_r) = (g_1v_1) \otimes (g_2v_2) \otimes \cdots \otimes (g_r v_r) \]
for $g_i \in O({n_i})$, $v_i \in \C^{n_i}$.

Let $\scP(\V)$ denote the algebra of complex valued polynomial functions on $\V$. The group $\K$ acts on $\scP(\V)$ in the standard way: for $k \in \K$, $v \in \V$, and $f \in \scP(\V)$, we have $k \cdot f(v) = f(k^Tv)$. The $\K$-invariant subspace of $\scP(V)$, which we denote by $\scP(\V)^{\K}$, is known to be finitely generated \cite{Hilbert}; however, a description of these invariants is incomplete outside of a few particular values of ${\bf n}$. 

We have the standard gradation $\scP(\V) = \bigoplus_{d = 0}^\infty \scP^d(\V)$, where $\scP^d(\V)$ is the subspace of degree $d$ homogeneous polynomials on $\V$. Moreover, the $\K$-invariant subalgebra inherits this gradation. That is, $\scP^{d}(\V)^{\K} = \scP^d(\V) \cap \scP(\V)^{\K}$. 

While $\dim \scP^{d}(\V)^{\K} $ is unknown in general, we show in section~\ref{rep_section} that the dimensions stabilize as the $n_i$'s are sufficiently large. It is in this setting we seek a formula for these dimensions, as well as a description of the invariants themselves. We first note that if $d$ is odd, we immediately obtain
\[ \dim \scP^d(\V)^{\K} = 0 \]
and so we assume $d = 2m$ where $m$ is a positive integer.

We call $\lambda$ a partition of size $d$, denoted $\lambda \vdash d$, if $\lambda = (\lambda_1 \geq \lambda_2 \geq \cdots \geq \lambda_\ell > 0)$ is a weakly decreasing and finite sequence of positive integers such that $\sum_{i=1}^\ell \lambda_i = d$. If $\lambda$ has $b_1$ ones, $b_2$ twos, $b_3$ threes, etc, we say that $\lambda$ has shape $(1^{b_1}2^{b_2}3^{b_3}\cdots)$. We define
\[ z_\lambda = 1^{b_1}2^{b_2}3^{b_3} \cdots b_1!b_2!b_3! \cdots \]
We let $\ell(\lambda)$ denote the length of the partition. Finally, we say $\lambda = (\lambda_1, \lambda_2, \ldots, \lambda_\ell)$ is even if $\lambda_i \in 2\Z$ for all $i$.

Let $\sigma$ be an element of $S_m$, the symmetric group on $m$ letters. When written in disjoint cycle notation, the lengths of the cycles of $\sigma$ form an integer partition of $m$, and we refer to the shape of $\sigma$ as above. We call a permutation $\tau$ a matching if it contains only cycles of length two; that is, $\tau$ is a fixed point free involution. 

In section~\ref{rep_section}, we show that
\begin{thm}\label{main_theorem}
Let $d$ be a positive, even integer and let $\K = O({n_1}) \times O({n_2}) \times \cdots \times O({n_r})$ and $\V =  \C^{n_1} \otimes \C^{n_2} \otimes \cdots \otimes \C^{n_r}$, where $n_1,n_2,\ldots,n_r \geq d$. Then dimension of the space of $\K$-invariant degree $d$ homogeneous polynomial functions on $\V$ is  
\[ \dim \scP^d(\V)^{\K} = \sum_{\lambda \vdash d} \frac{N(\lambda)^r}{z_\lambda} \] 
where $N(\lambda)$ is the number of matchings that commute with a permutation of shape $\lambda$.
\end{thm}

We determine the number of matchings which commute with a given permutation in section~\ref{combinatorics_section}. In particular, we have

\begin{thm}\label{Nlambda_theorem}
Given a permutation $\sigma$ with shape $\lambda = (1^{b_1}2^{b_2} \cdots t^{b_t})$, the number of matchings that commute with $\sigma$ is given by
\[ N(\lambda) = N((1^{b_1})) \cdot N((2^{b_2})) \cdot \cdots \cdot N((t^{b_t})) \] 
where 
\[ N((a^{b})) = \begin{cases}
				0 & \mbox{if } a \mbox{ odd and } b \mbox{ odd} \\
				\frac{b!a^{b/2}}{2^{b/2}(\frac{b}{2})!} & \mbox{if } a \mbox{ odd and } b \mbox{ even} \\
				\sum_{i=1, i \mbox{ \tiny odd}}^b \frac{b!a^{(b-i)/2}}{i!(\frac{b-i}{2})! 2^{(b-i)/2}} & \mbox{if } a \mbox{ even and } b \mbox{ odd} \\
				\sum_{i=0, i \mbox{ \tiny even}}^b \frac{b!a^{(b-i)/2}}{i!(\frac{b-i}{2})! 2^{(b-i)/2}} & \mbox{if } a \mbox{ even and } b \mbox{ even}
		\end{cases} \]  

\end{thm}

The notion of matchings has surprisingly deep connections to a variety of areas of research. Classical Schur-Weyl duality describes a relationship between irreducible finite dimensional representations of the general linear group and the symmetric group. The natural actions of each of these groups on the tensor space $\bigotimes^k(\C^n)$ centralize each other, resulting in the multiplicity free decomposition 
\[ \otimes^k\C^n \cong \mathop{\bigoplus_{\lambda \vdash k}}_{\ell(\lambda) \leq n} F^\lambda \otimes W^\lambda \]
where $F^\lambda$ and $W^\lambda$ are the irreducible representations of $GL_n$ and $S_n$, respectively, associated to the partition $\lambda$. In 1937, Richard Brauer defined an algebra which plays the role of the symmetric group in a similar statement on the representation theory of the orthogonal group \cite{Brauer}. For a partition $\lambda = (\lambda_1, \lambda_2, \ldots)$, let $\lambda' = (\lambda'_1, \lambda'_2, \ldots)$ denote the conjugate partition (thus $\lambda'_j$ is the number of boxes in the $j$th column of the Young diagram of $\lambda$). We have
\[ \otimes^k\C^n \cong \bigoplus^{\lfloor k/2 \rfloor}_{i = 0} \mathop{\bigoplus_{\lambda \vdash (k - 2i)}}_{\lambda'_1 + \lambda'_2 \leq n} V^\lambda \otimes U^\lambda \]
where $V^\lambda$ and $U^\lambda$ denote the respective irreducible representations of the Brauer algebra $\mathcal{B}_k(n)$ and $O(n)$. The dimension of the algebra $\mathcal{B}_k(n)$ is $(2k - 1)(2k - 3) \cdots 3 \cdots 1$,
the number of matchings on $2k$ elements. These basis elements of the algebra are frequently depicted as graphs; the diagram below represents a possible basis element of the algebra $\mathcal{B}_6(n)$:

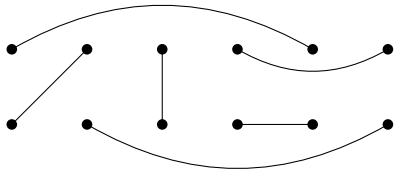
\begin{center}
\begin{tikzpicture}
\node[below] at (0,1) {};   \fill (0,1) circle (2pt);
\node[below] at (1,1) {};   \fill (1,1) circle (2pt);
\node[below] at (2,1) {};   \fill (2,1) circle (2pt);
\node[below] at (3,1) {};   \fill (3,1) circle (2pt);
\node[below] at (4,1) {};   \fill (4,1) circle (2pt);
\node[below] at (5,1) {};   \fill (5,1) circle (2pt);
\node[below] at (0,0) {};  \fill (0,0) circle (2pt);
\node[below] at (1,0) {};   \fill  (1,0) circle (2pt);
\node[below] at (2,0) {};   \fill  (2,0) circle (2pt);
\node[below] at (3,0) {};   \fill (3,0) circle (2pt);
\node[below] at (4,0) {};   \fill  (4,0) circle (2pt);
\node[below] at (5,0) {};   \fill (5,0) circle (2pt);
\draw(0,1) edge[bend left] (4,1);
\draw(0,0) edge (1,1);
\draw(1,0) edge[bend right] (5,0);
\draw(2,0) edge (2,1);
\draw(3,0) edge (4,0);
\draw(3,1) edge[bend right] (5,1);
\end{tikzpicture}
\end{center}

The number of matchings on a set of elements has applications outside of representation theory as well. Numerous fields in biology use graphs called phylogenetic trees to illustrate inferred evolutionary relationships. These graphs are full rooted binary trees in which each node, with the exception of the leaves, has exactly two children. Each node could, for example, represent a particular species. A parent node (that is, one that is not a leaf) would represent the most recent genetic ancestor of the two child species. Diaconis and Holmes \cite{Diaconis} have described a bijection between matchings on $2n$ elements and phylogenetic trees with $n+1$ leaves, which we recount here. 

We first describe the procedure for building a tree with $n+1$ labelled leaves from a matching on the set $\{1, 2, \ldots, 2n\}$, which we illustrate with the example $\tau = (1 \ 4)(2 \ 3)(5 \ 8)(6 \ 7)$.  The resulting tree will have leaves labelled with the set $\{1, 2, \ldots, n+1\}$; the ancestors of the tree will be labelled $\{n+2, n+3, \ldots, 2n\}$. First, look for any pairs in the matching that contain only numbers at most $n+1$. In our specific example, we have two such cycles, $(1 \ 4)$ and $(2 \ 3)$. Each such cycle defines a sibling pair of leaves. Choose the pair with the smallest child, in this case, $(1 \ 4)$. Label the parent of this pair with the smallest  ancestor, 6. The other sibling pair will be labelled with the next available ancestor, 7. Since 6 and 7 appear in the same cycle of the matching, they must also be siblings in the tree. Their parent will be the last available ancestor, 8. Finally, 8 must be paired with 5, which will form its own leaf. The resulting tree is shown on the left below. The corresponding tree with labelled leaves is on the right.
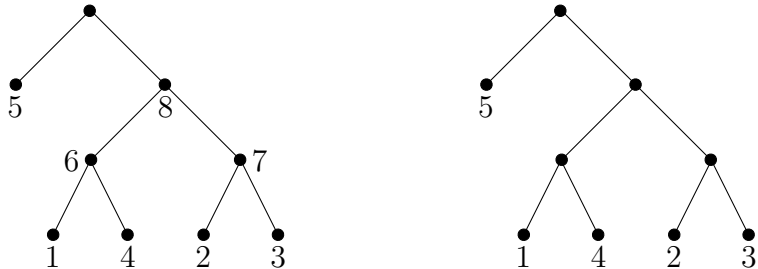
\begin{center}
\begin{tikzpicture}
      	\node[below] at (0,0) {1};
	\node[below] at (1,0) {4};
	\node[below] at (2,0) {2};
	\node[below] at (3,0)  {3};
	\node[left] at (0.5,1) {6};
	\node[right] at (2.5,1)  {7};
	\node[below] at (1.5,2) {8};
	\node[below] at (-0.5,2)  {5};

	\draw[*-*, shorten >=-2, shorten <=-2] (0.5,3) to (-0.5,2);
	\draw[-*, shorten >=-2] (0.5,3) to (1.5,2);
	\draw[-*, shorten >=-2] (1.5,2) to (0.5,1);
	\draw[-*, shorten >=-2] (1.5,2) to (2.5,1);
	\draw[-*, shorten >=-2] (2.5,1) to (3,0);
	\draw[-*, shorten >=-2] (2.5,1) to (2,0);
	\draw[-*, shorten >=-2] (0.5,1) to (0,0);
	\draw[-*, shorten >=-2] (0.5,1) to (1,0);
\end{tikzpicture}
\hspace{2cm}
\begin{tikzpicture}
      	\node[below] at (0,0) {1};
	\node[below] at (1,0) {4};
	\node[below] at (2,0) {2};
	\node[below] at (3,0)  {3};
	\node[left] at (0.5,1) {};
	\node[right] at (2.5,1)  {};
	\node[below] at (1.5,2) {};
	\node[below] at (-0.5,2)  {5};

	\draw[*-*, shorten >=-2, shorten <=-2] (0.5,3) to (-0.5,2);
	\draw[-*, shorten >=-2] (0.5,3) to (1.5,2);
	\draw[-*, shorten >=-2] (1.5,2) to (0.5,1);
	\draw[-*, shorten >=-2] (1.5,2) to (2.5,1);
	\draw[-*, shorten >=-2] (2.5,1) to (3,0);
	\draw[-*, shorten >=-2] (2.5,1) to (2,0);
	\draw[-*, shorten >=-2] (0.5,1) to (0,0);
	\draw[-*, shorten >=-2] (0.5,1) to (1,0);
\end{tikzpicture}
\end{center}
Note that there is more than one possible rule to create such a correspondence.

We now sketch a method of finding a matching when given a tree. We start with a binary rooted tree with $n+1$ labelled leaves. We look for the pair of siblings with the smallest child, and label the parent of these children with $n+2$. Repeat this process until all nodes (except the root) have been labelled. The matching defined by this tree is formed by pairing siblings into cycles. A particular example is shown below. The tree on the left is a binary tree with 7 leaves. The center diagram is obtained by labeling the nodes of the tree as described, and the corresponding permutation is on the right.

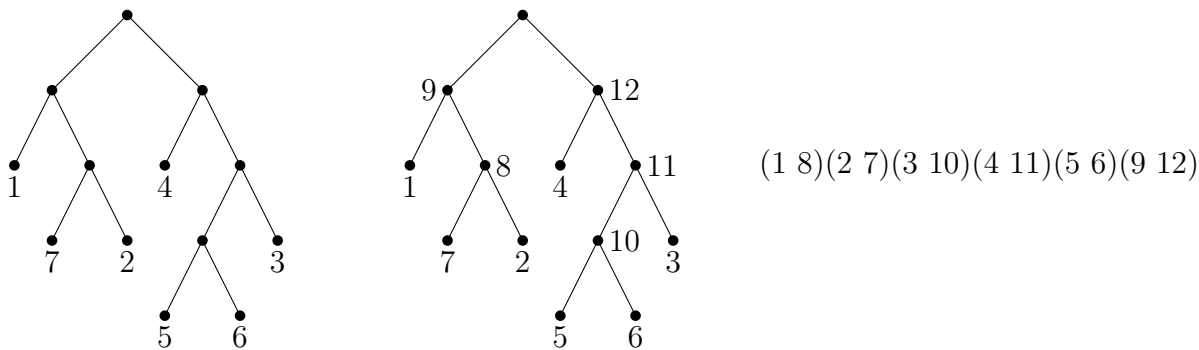
\begin{center}
\begin{tikzpicture}
\node[below] at (0,0) {5}; \fill (0,0) circle (2pt);
\node[below] at (1,0) {6}; \fill (1,0) circle (2pt);
\fill (0.5,1) circle (2pt);
\node[below] at (-0.5,1) {2}; \fill (-0.5,1) circle (2pt);
\node[below] at (-1.5,1) {7}; \fill (-1.5,1) circle (2pt);
\fill (-1,2) circle (2pt);
\node[below] at (-2,2) {1}; \fill (-2,2) circle (2pt);
\fill (-1.5,3) circle (2pt);
\node[below] at (1.5,1) {3}; \fill (1.5,1) circle (2pt);
\fill (1,2) circle (2pt);
\node[below] at (0,2) {4}; \fill (0,2) circle (2pt);
\fill (0.5,3) circle (2pt);
\fill (-0.5,4) circle (2pt);

\draw (0,0) edge (0.5,1);
\draw (1,0) edge (0.5,1);
\draw (-0.5,1) edge (-1,2);
\draw (-1.5,1) edge (-1,2);
\draw (-1.5,3) edge (-2,2);
\draw (-1.5,3) edge (-1,2);
\draw(1,2) edge (1.5,1);
\draw(1,2) edge (0.5,1);
\draw(0,2) edge (0.5,3);
\draw(0.5,3) edge (1,2);
\draw(0.5,3) edge (-0.5,4);
\draw(-1.5,3) edge (-0.5,4);
\end{tikzpicture}
\hspace{1cm}
\begin{tikzpicture}
\node[below] at (0,0) {5}; \fill (0,0) circle (2pt);
\node[below] at (1,0) {6}; \fill (1,0) circle (2pt);
\node[right] at (0.5,1) {10}; \fill (0.5,1) circle (2pt);
\node[below] at (-0.5,1) {2}; \fill (-0.5,1) circle (2pt);
\node[below] at (-1.5,1) {7}; \fill (-1.5,1) circle (2pt);
\node[right] at (-1,2) {8}; \fill (-1,2) circle (2pt);
\node[below] at (-2,2) {1}; \fill (-2,2) circle (2pt);
\node[left] at (-1.5,3) {9}; \fill (-1.5,3) circle (2pt);
\node[below] at (1.5,1) {3}; \fill (1.5,1) circle (2pt);
\node[right] at (1,2) {11}; \fill (1,2) circle (2pt);
\node[below] at (0,2) {4}; \fill (0,2) circle (2pt);
\node[right] at (0.5,3) {12}; \fill (0.5,3) circle (2pt);
\fill (-0.5,4) circle (2pt);

\draw (0,0) edge (0.5,1);
\draw (1,0) edge (0.5,1);
\draw (-0.5,1) edge (-1,2);
\draw (-1.5,1) edge (-1,2);
\draw (-1.5,3) edge (-2,2);
\draw (-1.5,3) edge (-1,2);
\draw(1,2) edge (1.5,1);
\draw(1,2) edge (0.5,1);
\draw(0,2) edge (0.5,3);
\draw(0.5,3) edge (1,2);
\draw(0.5,3) edge (-0.5,4);
\draw(-1.5,3) edge (-0.5,4);

\node[right] at (2.5,2) {$(1 \ 8)(2 \ 7)(3 \ 10)(4 \ 11)(5 \ 6)(9 \ 12)$};
\end{tikzpicture}

\end{center}

In section~\ref{graph_section}, we show that given an $r$-tuple of matchings $(\tau_1, \tau_2, \ldots, \tau_r)$, each on $2m$ letters, we can describe an invariant in $\scP^{2m}(\V)^{\K}$. The symmetric group $S_{2m}$ acts on such a tuple by simultaneous conjugation; that is, for $g \in S_{2m}$, we define
\[ g \cdot (\tau_1, \tau_2, \ldots, \tau_r) = (g\tau_1 g^{-1}, g\tau_2 g^{-1}, \ldots, g\tau_r g^{-1}) \]
We show that the orbits of this action are associated to a generator of the algebra  $\scP^{2m}(\V)^{\K}$, and are in bijection to the isomorphism classes of certain graphs. 

We can therefore consider an action of $S_{2m}$ on $r$-tuples of phylogenetic trees with $m+1$ leaves. Using the method described above, we  write an $r$-tuple of trees as an $r$-tuple of matchings, apply the action of simultaneous conjugation by an element of $S_{2m}$, and then draw the list of trees associated to the resulting $r$-tuple of matchings.  Hence, we define an action of $S_{2m}$ on a forest of $r$ phylogenetic trees, each with $m+1$ leaves.

To illustrate, suppose we choose a tuple of three trees, each with four leaves. Labeling the roots of the trees 1 through 3 and adjoining these roots to a common vertex creates a forest:
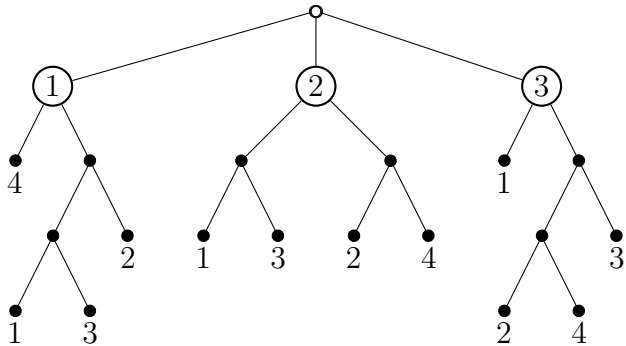
\begin{center}
\begin{tikzpicture}
      	\node (1) at (0.5,3) [circle,draw, thick,inner sep=1.8pt] {1};
      	\node[below] at (0,0) {1};
	\node[below] at (1,0) {3};
	\node[below] at (1.5,1) {2};
	\node[below] at (0,2)  {4};
	\draw[-*, shorten >=-2] (1) to (1,2);
	\draw[-*, shorten >=-2] (1) to (0,2);
	\draw[-*, shorten >=-2] (1,2) to (1.5,1);
	\draw[-*, shorten >=-2] (1,2) to (0.5,1);
	\draw[-*, shorten >=-2] (0.5,1) to (0,0);
	\draw[-*, shorten >=-2] (0.5,1) to (1,0);
	
	\node (2) at (4,3) [circle,draw, thick,inner sep=1.8pt] {2};
      	\node[below] at (2.5,1) {1};
	\node[below] at (3.5,1) {3};
	\node[below] at (4.5,1) {2};
	\node[below] at (5.5,1)  {4};
	\draw[-*, shorten >=-2] (2) to (5,2);
	\draw[-*, shorten >=-2] (2) to (3,2);
	\draw[-*, shorten >=-2] (5,2) to (5.5,1);
	\draw[-*, shorten >=-2] (5,2) to (4.5,1);
	\draw[-*, shorten >=-2] (3,2) to (3.5,1);
	\draw[-*, shorten >=-2] (3,2) to (2.5,1);
	
	\node (3) at (7,3) [circle,draw, thick,inner sep=1.8pt] {3};
      	\node[below] at (6.5,0) {2};
	\node[below] at (7.5,0) {4};
	\node[below] at (8,1) {3};
	\node[below] at (6.5,2)  {1};
	\draw[-*, shorten >=-2] (3) to (7.5,2);
	\draw[-*, shorten >=-2] (3) to (6.5,2);
	\draw[-*, shorten >=-2] (7.5,2) to (8,1);
	\draw[-*, shorten >=-2] (7.5,2) to (7,1);
	\draw[-*, shorten >=-2] (7,1) to (6.5,0);
	\draw[-*, shorten >=-2] (7,1) to (7.5,0);
	
	\draw[-o, shorten >=-2] (1) to (4,4);
	\draw[-o, shorten >=-2] (2) to (4,4);
	\draw[-o, shorten >=-2] (3) to (4,4);
\end{tikzpicture}
\end{center}
By following the process outlined earlier, we associate this forest to the 3-tuple of matchings
\[ (\tau_1, \tau_2, \tau_3) = ((1 \ 3)(2 \ 5)(4 \ 6), (1 \ 3)(2 \ 4)(5 \ 6), (1 \ 6)(2 \ 4)(3 \ 5)) \]
We choose a permutation $g = (1 \ 3 \ 5)(2 \ 4)(6)$ to act on this tuple:
\[ (g\tau_1g^{-1}, g\tau_2g^{-1}, g\tau_3g^{-1}) = ((1 \ 4)(2 \ 6)(3 \ 5), (1 \ 6)(2 \ 4)(3 \ 5), (1 \ 5)(2 \ 4)(3 \ 6)) \]
and draw the forest associated to the result:
\smallskip
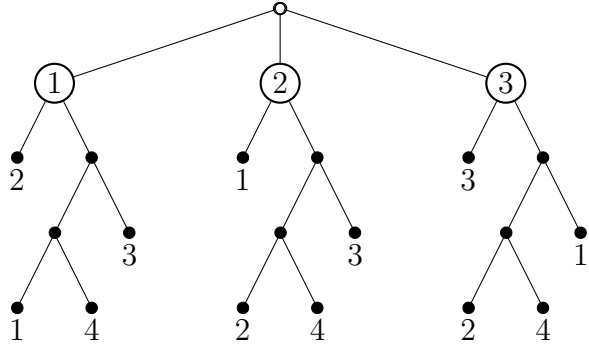
\begin{center}
\begin{tikzpicture}
      	\node (1) at (0.5,3) [circle,draw, thick,inner sep=1.8pt] {1};
      	\node[below] at (0,0) {1};
	\node[below] at (1,0) {4};
	\node[below] at (1.5,1) {3};
	\node[below] at (0,2)  {2};
	\draw[-*, shorten >=-2] (1) to (1,2);
	\draw[-*, shorten >=-2] (1) to (0,2);
	\draw[-*, shorten >=-2] (1,2) to (1.5,1);
	\draw[-*, shorten >=-2] (1,2) to (0.5,1);
	\draw[-*, shorten >=-2] (0.5,1) to (0,0);
	\draw[-*, shorten >=-2] (0.5,1) to (1,0);
	
	\node (2) at (3.5,3) [circle,draw, thick,inner sep=1.8pt] {2};
      	\node[below] at (3,0) {2};
	\node[below] at (4,0) {4};
	\node[below] at (3,2) {1};
	\node[below] at (4.5,1)  {3};
	\draw[-*, shorten >=-2] (2) to (3,2);
	\draw[-*, shorten >=-2] (2) to (4,2);
	\draw[-*, shorten >=-2] (4,2) to (4.5,1);
	\draw[-*, shorten >=-2] (4,2) to (3.5,1);
	\draw[-*, shorten >=-2] (3.5,1) to (4,0);
	\draw[-*, shorten >=-2] (3.5,1) to (3,0);
	
	\node (3) at (6.5,3) [circle,draw, thick,inner sep=1.8pt] {3};
      	\node[below] at (6,0) {2};
	\node[below] at (7,0) {4};
	\node[below] at (7.5,1) {1};
	\node[below] at (6,2)  {3};
	\draw[-*, shorten >=-2] (3) to (7,2);
	\draw[-*, shorten >=-2] (3) to (6,2);
	\draw[-*, shorten >=-2] (7,2) to (7.5,1);
	\draw[-*, shorten >=-2] (7,2) to (6.5,1);
	\draw[-*, shorten >=-2] (6.5,1) to (6,0);
	\draw[-*, shorten >=-2] (6.5,1) to (7,0);
	
	\draw[-o, shorten >=-2] (1) to (3.5,4);
	\draw[-o, shorten >=-2] (2) to (3.5,4);
	\draw[-o, shorten >=-2] (3) to (3.5,4);
\end{tikzpicture}
\end{center}

In \cite{HWW}, a bijection is described between the invariants defined on the tensor space  $\otimes_{i = 1}^r \C^{n_i}$ under the action of a product of unitary groups and isomorphism classes of finite coverings of connected simple graphs. In section~\ref{graph_section}, we seek a similar graphical interpretation of the invariants on the same space under the action of orthogonal groups. Specifically, we produce formulas for the invariants as well as a bijection between a basis for the space of invariants and isomorphism classes of edge-colored $r$-regular graphs on $d$ vertices.

\bigskip

\section{Proof of Theorem~\ref{main_theorem}}\label{rep_section}

Let $m$ be a positive integer and let $d = 2m$. For each partition $\lambda$ of $d$ with at most $n$ parts, there exists an irreducible representation of $GL_n(\C)$ with highest weight indexed by $\lambda$, which we denote by $F^\lambda_n$ following \cite{GW}. If $\lambda(i) \vdash d$ for $1 \leq i \leq r$, then  
\[ F^{\lambda^{(1)}}_{n_1} \otimes F^{\lambda^{(2)}}_{n_2}  \otimes \cdots \otimes F^{\lambda^{(r)}}_{n_r} \]
is an irreducible representation of $G({\bf n})$ which embeds in $\scP^d(\V)$ with multiplicity denoted by $g_{\lambda^{(1)}\lambda^{(2)}\cdots\lambda^{(r)}}$. That is,
\[ \scP^d(\V) \cong \mathop{\bigoplus_{\lambda^{(i)} \vdash d}}_{\ell(\lambda^{(i)}) \leq n_i} g_{\lambda^{(1)}\lambda^{(2)}\cdots\lambda^{(r)}} F^{\lambda^{(1)}}_{n_1} \otimes F^{\lambda^{(2)}}_{n_2}  \otimes \cdots \otimes F^{\lambda^{(r)}}_{n_r} \]
We can now write the $\K$-invariants as
\begin{align*}
\left[\scP^d(\V)\right]^{\K} &\cong \mathop{\bigoplus_{\lambda^{(i)} \vdash d}}_{\ell(\lambda^{(i)}) \leq n_i} g_{\lambda^{(1)}\lambda^{(2)}\cdots\lambda^{(r)}} \left( F^{\lambda^{(1)}}_{n_1} \otimes F^{\lambda^{(2)}}_{n_2}  \otimes \cdots \otimes F^{\lambda^{(r)}}_{n_r}\right)^{\K} \\
& \cong \mathop{\bigoplus_{\lambda^{(i)} \vdash d}}_{\ell(\lambda^{(i)}) \leq n_i} g_{\lambda^{(1)}\lambda^{(2)}\cdots\lambda^{(r)}} \left( F^{\lambda^{(1)}}_{n_1}\right)^{O({n_1})} \otimes \left(F^{\lambda^{(2)}}_{n_2}\right)^{O({n_2})}  \otimes \cdots \otimes \left( F^{\lambda^{(r)}}_{n_r}\right)^{O({n_r})}
\end{align*}

The Cartan-Helgason Theorem \cite{GW} tells us that if $F^\lambda_n$ is an irreducible representation of $GL_n(\C)$, then $\dim\left(F^\lambda_n\right)^{O(n)}$ is at most one. In particular, we have
\[ \dim\left(F^\lambda_n \right)^{O(n)} = \begin{cases} 1 & \mbox{if $\lambda$ is even} \\ 0 & \mbox {otherwise} \end{cases} \]
and so
\begin{equation} \label{first_g_sum}
\dim \left[\scP^d(\V)\right]^{\K}  = \mathop{ \mathop{\sum_{\lambda^{(1)}, \ldots, \lambda^{(r)}}}_{\lambda^{(i)} \vdash 2m}}_{\lambda^{(i)} \mbox{ \tiny even}}  g_{\lambda^{(1)}\lambda^{(2)}\cdots\lambda^{(r)}}
\end{equation}
The dimensions will stabilize when the sum is taken over all even partitions of $2m$; that is, when $n_1, n_2, \ldots, n_r \geq 2m$. 

We next recall Schur-Weyl duality, which will allow us to find another interpretation of Equation~\ref{first_g_sum}. While $GL_n$ acts on the space $\otimes^m \C^n$ by simultaneous matrix multiplication, the symmetric group acts on the same space by permuting tensor factors. That is, given $x \in GL_n$ and a permutation $\sigma \in S_m$, we have
\[ x (v_1 \otimes v_2 \otimes \cdots \otimes v_m) = xv_1 \otimes xv_2 \otimes \cdots \otimes xv_m \]
and
\[ \sigma (v_1 \otimes v_2 \otimes \cdots \otimes v_m) = v_{\sigma(1)} \otimes v_{\sigma(2)}\otimes \cdots \otimes v_{\sigma(m)} \]
for $v_1, v_2, \ldots , v_m \in \C^n$. The subalgebras of End($\otimes^m\C^n$) generated by each of these actions are full mutual commutants. Hence, we obtain a multiplicity free decomposition
\[ \otimes^m\C^n \cong F^\lambda_n \otimes U^\lambda_m \]
where $U^\lambda_m$ is an irreducible representation of $S_m$ indexed by the partition $\lambda$.

Let $U^\lambda$ denote the irreducible representation of the symmetric group $S_{2m}$ indexed by the partition $\lambda$. We have
\begin{equation}\label{S_2m}
\dim( U^{\lambda(1)} \otimes  U^{\lambda(2)} \otimes \cdots \otimes  U^{\lambda(r)})^{S_{2m}} = g_{\lambda^{(1)}\lambda^{(2)}\cdots\lambda^{(r)}}
\end{equation}
where each $\lambda(i) \vdash 2m$. These are in fact the same multiplicities that appear in the decomposition of representations of the general linear group \cite[Prop. 3]{HeroWillenbring}. The numbers $g_{\lambda(1)\lambda(2)\cdots\lambda(r)} $ are known as the Kronecker coefficients, and are known to be hard to compute. 

Fix $\tau_0 = (1 \ \ 2)(3 \ \ 4) \cdots (2m-1 \ \ 2m)$ in $S_{2m}$. Denote by $H_m$ the centralizer\footnote{$H_m$ is isomorphic to the hyperoctahedral group, the group of symmetries of a hypercube of dimension $m$.} of $\tau_0$ in $S_{2m}$. We note that $(S_{2m}, H_m)$ is a Gelfand pair \cite{Stanley}, and so $\dim(U^{\lambda})^{H_m} $ is at most one. In fact, we have
\[ \dim(U^{\lambda})^{H_m} = \begin{cases} 1 & \mbox{if } \lambda \mbox{ is even} \\
							0 & \mbox{otherwise} \end{cases} \]
Thus, 
\begin{equation}\label{H_m}
 \dim[ U^{\lambda(1)} \otimes  U^{\lambda(2)} \otimes \cdots \otimes  U^{\lambda(r)} ]^{H_m \times H_m \times \cdots \times H_m} = \begin{cases} 1 & \mbox{if } \lambda(i) \mbox{ is even for all }  i \in \{1, 2, \ldots, r\} \\
							0 & \mbox{otherwise} \end{cases} 
\end{equation}
As an immediate consequence of Peter-Weyl decomposition, we have
\[ \mathbb{C}[S_{2m}] \cong \bigoplus_{\lambda} U^\lambda \otimes U^\lambda \]
We can now describe the group algebra of a product of $r$ copies of $S_{2m}$ as follows:
\begin{align*}\label{alg}
\mathbb{C}[S_{2m} \times S_{2m} \times \cdots \times S_{2m}] &\cong  \mathbb{C}[S_{2m}] \otimes \mathbb{C}[S_{2m}] \otimes \cdots \otimes \mathbb{C}[S_{2m}]  \\
			&\cong \left( \bigoplus_{\lambda} U^\lambda \otimes U^\lambda \right) \otimes \left( \bigoplus_{\lambda} U^\lambda \otimes U^\lambda \right) \otimes \cdots \otimes \left( \bigoplus_{\lambda} U^\lambda \otimes U^\lambda \right) \\
			&\cong \bigoplus_{\lambda^{(1)},\ldots,\lambda^{(r)}} \left(  U^{\lambda^{(1)}} \otimes U^{\lambda^{(1)}} \right) \otimes \cdots \otimes \left(  U^{\lambda^{(r)}} \otimes U^{\lambda^{(r)}} \right)
\end{align*}
Combining this fact with Equations \ref{S_2m} and \ref{H_m}, we obtain
\begin{align*}
\dim \mathbb{C}[S_{2m} &\times S_{2m} \times \cdots \times S_{2m}]^{S_{2m} \times H_m \times \cdots \times H_m} \\
	&= \mathop{\sum_{\lambda^{(1)}, \ldots, \lambda^{(r)}}}_{\lambda^{(i)} \vdash 2m} \dim\left[ \left(U^{\lambda^{(1)}} \otimes \cdots \otimes U^{\lambda^{(r)}}\right)^{S_{2m}} \otimes \left(U^{\lambda^{(1)}} \otimes \cdots \otimes U^{\lambda^{(r)}}\right)^{H_m \times \cdots \times H_m}  \right] \\
	&= \mathop{ \mathop{\sum_{\lambda^{(1)}, \ldots, \lambda^{(r)}}}_{\lambda^{(i)} \vdash 2m}}_{\lambda^{(i)} \mbox{ even}}  g_{\lambda^{(1)}\lambda^{(2)}\cdots\lambda^{(r)}}
\end{align*}

So far, we have shown that
\[ \dim \left[\scP^d(\V)\right]^{\K} = \dim \mathbb{C}[S_{2m} \times S_{2m} \times \cdots \times S_{2m}]^{S_{2m} \times H_m \times \cdots \times H_m}\]
Define the set
\[ \tI = \{ \tau \in S_{2m} \ : \ \tau^2 = \mbox{id}, \tau(i) \neq i \mbox{ for all } i \leq n \} \]
That is, $\tI$ is the set of matchings on $2m$ letters. As an $S_{2m}$-set, we have $\tI \cong S_{2m}/H_m$. We will denote the product of $r$ copies of $\tI$ by $\tI^r$. 

Recall that given a group $G$ and $H \subset G$, we have $\dim \mathbb{C}[G]^H = \dim \mathbb{C}[G/H]$. Thus, by the previous results
\begin{align*}
\mathop{ \mathop{\sum_{\lambda^{(1)}, \ldots, \lambda^{(r)}}}_{\lambda^{(i)} \vdash 2m}}_{\lambda^{(i)} \mbox{\footnotesize even}}  g_{\lambda^{(1)}\lambda^{(2)}\cdots\lambda^{(r)}}
&= \dim \mathbb{C}[S_{2m} \times S_{2m} \times \cdots \times S_{2m}]^{S_{2m} \times H_m \times \cdots \times H_m} \\
&= \dim \mathbb{C} [S_{2m}/H_m \times \cdots \times S_{2m}/H_m ]^{S_{2m}} \\
&= \dim \mathbb{C} [\tI^r]^{S_{2m}}
\end{align*}

Given a group $G$ and a set $X$, it is well known that the dimension of the space of invariants $\mathbb{C}[X]^G$ is equal to the number of orbits of the action of $G$ on $X$. By Burnside's Lemma, the number of such orbits is the average number of $x \in X$ fixed by $g \in G$; that is,
\[ \dim \mathbb{C}[X]^G = \frac{1}{|G|} \sum_{g \in G} |X|^g \]
where $|X|^g$ denotes the cardinality of the set of points in $X$ fixed by $g$. In our setting, where $G = S_{2m}$ and $X = \tI^r$, we have
\begin{equation*} \label{burnside}
 \dim \mathbb{C} [\tI^r ]^{S_{2m}} = \frac{1}{|S_{2m}|} \sum_{g \in S_{2m}} |\tI^r|^g
\end{equation*}
where $S_{2m}$ acts on $\tI^r$ by simultaneous conjugation. That is, given $(\tau_1, \tau_2, \ldots, \tau_r) \in \tI^r$, $g \in S_{2m}$, we define the action
\[ g.(\tau_1, \tau_2, \ldots, \tau_r) = (g\tau_1g^{-1}, g\tau_2g^{-1}, \ldots, g\tau_rg^{-1}) \]
We easily see that
\[ |\tI^r|^g = \left(|\tI|^g\right)^r \]
and so, given $g \in S_{2m}$,  it remains only to find a formula for the number of matchings $\tau \in \tI$ such that $g \tau g^{-1} = \tau$. It is easily seen that this number is the same for two permutations with the same cycle type, as we can simply relabel the entries of each cycle. Thus, if $g,h \in S_{2m}$ have the same cycle type, we have $|\tI|^g = |\tI|^h$.

Recall that two elements in $S_{2m}$ are conjugate if they have the same cycle type. Denote by $\widehat{S_{2m}}$ the set of conjugacy classes in $S_{2m}$, indexed by integer partitions $\mu$ of $2m$. Hence we can define a class function $N: \widehat{S_{2m}} \to \mathbb{N}$ by setting $N(\lambda)$ equal to the number of matchings $\tau \in \tI$ that commute with a permutation $g \in S_{2m}$ with cycle type $\mu$. Finally, we have shown that for $\V = \C^{n_1} \otimes \C^{n_2} \otimes \cdots \otimes \C^{n_r}$, $\K = O({n_1}) \times O({n_2}) \times \cdots \times O({n_r})$ and $n_i \geq 2m$ for $1 \leq i \leq r$, we have
\begin{equation*} \label{new}
\dim \left[\scP^{2m}(\V)\right]^{\K} =  \dim \mathbb{C} [\tI^r ]^{S_{2m}} = \frac{1}{|S_{2m}|} \sum_{g \in S_{2m}} |\tI^r|^g = \frac{1}{(2m)!} \sum_{\lambda \vdash 2m}N(\lambda)^r
\end{equation*}
The formula for $N(\lambda)$ is presented in Theorem~\ref{Nlambda_theorem}, which we prove in the next section.

\bigskip

\section{Proof of Theorem~\ref{Nlambda_theorem}} \label{combinatorics_section}

We now determine the formula for $N(\lambda)$, the number of matchings which commute with a permutation with cycle type $\lambda$. Our strategy will be to focus initially on permutations with shape $(m^2)$. We then consider ``brick'' permutations with shape $(a^b)$, where $ab = 2m$, and finally we generalize to all permutations.

\subsection{The Case $\lambda = (m^2)$}

We begin with:

\begin{lemma} \label{brick_matchings}
Let $g = (\alpha_1 \ \alpha_2 \cdots \ \alpha_m)(\beta_1 \ \beta_2 \ \cdots \beta_m)$ be a product of two cycles of equal length $m$. If $\sigma$ is a matching that commutes with $g$, then either $\sigma = g^{m/2}$ or $\sigma$ has form
\[ \sigma = (\alpha_1 \ \beta_{j})(\alpha_2 \ \beta_{j+1})(\alpha_3 \ \beta_{j+2}) \cdots (\alpha_m  \ \beta_{j-1}) \]
where $0 \leq i < m$. Hence, if $m$ is even there are exactly $m+1$ matchings that commute with $g$, and $m$ such matchings if $m$ is odd.

\end{lemma}

\begin{proof}
Suppose first that $\sigma$ permutes the elements within each cycle of $g$. Let $\sigma(\alpha_k) = \alpha_m$ for some $k$ with $1 \leq k \leq m$; that is, assume the transposition $(\alpha_k \ \alpha_m)$ appears in $\sigma$. Then
\[ \alpha_{k+1} = g(\alpha_{k}) = \sigma g \sigma(\alpha_{k}) = \sigma g(\alpha_m) = \sigma(\alpha_1) \]
and so $(\alpha_{1} \ \alpha_{k+1})$ is a transposition in $\sigma$. Similarly,
\[ \alpha_{k+2} = g(\alpha_{k+1}) = \sigma g \sigma(\alpha_{k+1}) = \sigma g(\alpha_1) = \sigma(\alpha_2) \]
and so $(\alpha_{2} \ \alpha_{k+2})$ is a transposition in $\sigma$. Continuing in this way, we find $\sigma$ contains the transpositions
$(\alpha_{1} \ \alpha_{k+1}), (\alpha_{2} \ \alpha_{k+2}), (\alpha_{3} \ \alpha_{k+3}), \ldots, (\alpha_k \ \alpha_m) $.
Now
\[ \sigma(\alpha_m) \alpha_k = g(\alpha_{k-1}) = \sigma g \sigma(\alpha_{k-1}) = \sigma g(\alpha_{2k-1}) = \sigma(\alpha_{2k}) \]
and hence $m = 2k$. Then we have only one possible choice for $k$, $k = m/2$. Thus
\[ \sigma = (\alpha_{1} \ \alpha_{\frac{m}{2}+1})(\alpha_{2} \ \alpha_{\frac{m}{2}+2})(\alpha_{3} \ \alpha_{\frac{m}{2}+3}) \cdots (\alpha_{\frac{m}{2}} \ \alpha_m) = (\alpha_1 \ \alpha_2 \cdots \ \alpha_m)^{m/2} \]
Thus if $\sigma$ permutes the elements within the cycles of $g$, we must have $\sigma = g^{m/2}$. Note such a $\sigma$ only exists if $m$ is even.

Now suppose $\sigma$ interchanges elements between the two cycles of $g$. Assume $\sigma(\alpha_1) = \beta_j$ for some $j$ with $1 \leq j \leq m$. Then $(\alpha_1 \ \beta_j)$ is a transposition appearing in $\sigma$. Since $\sigma g \sigma = g$, we have
\[ \alpha_2 = g(\alpha_1) = \sigma g \sigma (\alpha_1) = \sigma g(\beta_j) = \sigma(\beta_{j+1}) \]
and so $(\alpha_2 \ \beta_{j+1})$ appears in $\sigma$. Similarly,
\[ \alpha_3 = g(\alpha_2) = \sigma g \sigma (\alpha_2) = \sigma g(\beta_{j+1}) = \sigma(\beta_{j+2}) \]
and so $(\alpha_3 \ \beta_{j+2})$ appears in $\sigma$. Continuing in this way, we have
\[ \sigma = (\alpha_1 \ \beta_j)(\alpha_2 \ \beta_{j+1})(\alpha_3 \ \beta_{j+2}) \cdots (\alpha_m \ \beta_{j-1}) \]
\end{proof}

The matchings that will commute with a fixed permutation with two cycles of equal length can easily be interpreted as diagrams. With $g$ as in the above proof, we draw two rows of $m$ nodes. We label the nodes along the top row with $\alpha_1, \alpha_2, \ldots, \alpha_m$, and along the bottom row with $\beta_1, \beta_2, \ldots, \beta_m$. Drawing edges so that each node is connected to exactly one other node defines a matching. When we consider only the diagrams corresponding to matchings which commute with $g$, an immediate pattern emerges. If $m$ is even, the first such diagram is obtained by drawing an edge from $\alpha_1$ to $\alpha_{\frac{m}{2} + 1}$, an edge from $\alpha_2$ to $\alpha_{\frac{m}{2} + 2}$, and so on. The nodes labelled $\beta_i$ are matched in an identical way, resulting in a diagram which represents the matching $g^{m/2}$. The remaining $m$ diagrams (for odd or even values of $m$) are obtained by drawing edges from $\alpha_i$ to $\beta_i$ for each $i$, and then cyclically permuting the second row of unlabeled nodes. The five matchings that commute with $g = (\alpha_1 \ \alpha_2 \ \alpha_3 \ \alpha_4)(\beta_1 \ \beta_2 \ \beta_3 \ \beta_4)$, for instance, are:

\bigskip

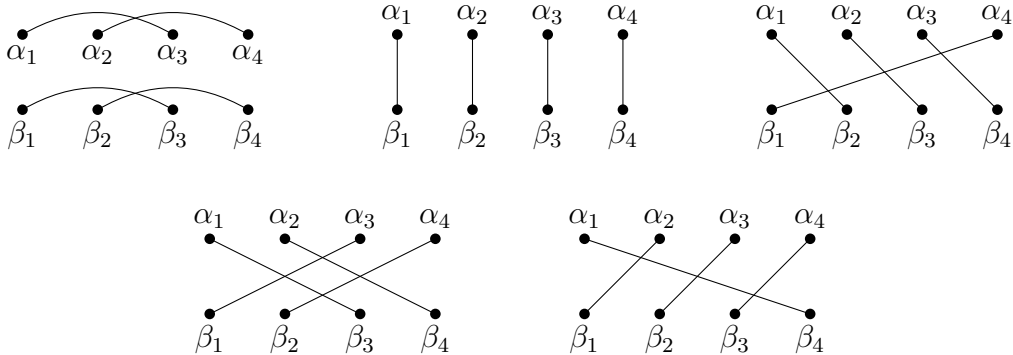
\begin{center}
\begin{tikzpicture}
\node[below] at (0,1) {$\alpha_1$};   \fill (0,1) circle (2pt);
\node[below] at (1,1) {$\alpha_2$};   \fill (1,1) circle (2pt);
\node[below] at (2,1) {$\alpha_3$};   \fill (2,1) circle (2pt);
\node[below] at (3,1) {$\alpha_4$};   \fill (3,1) circle (2pt);

\node[below] at (0,0) {$\beta_1$};  \fill (0,0) circle (2pt);
\node[below] at (1,0) {$\beta_2$};   \fill  (1,0) circle (2pt);
\node[below] at (2,0) {$\beta_3$};   \fill  (2,0) circle (2pt);
\node[below] at (3,0) {$\beta_4$};   \fill (3,0) circle (2pt);

\draw(0,1) edge[bend left] (2,1);
\draw(1,1) edge[bend left] (3,1);
\draw(0,0) edge[bend left] (2,0);
\draw(1,0) edge[bend left] (3,0);
\end{tikzpicture}
\hspace{1cm}
\begin{tikzpicture}
\node[above] at (0,1) {$\alpha_1$};   \fill (0,1) circle (2pt);
\node[above] at (1,1) {$\alpha_2$};   \fill (1,1) circle (2pt);
\node[above] at (2,1) {$\alpha_3$};   \fill (2,1) circle (2pt);
\node[above] at (3,1) {$\alpha_4$};   \fill (3,1) circle (2pt);

\node[below] at (0,0) {$\beta_1$};  \fill (0,0) circle (2pt);
\node[below] at (1,0) {$\beta_2$};   \fill  (1,0) circle (2pt);
\node[below] at (2,0) {$\beta_3$};   \fill  (2,0) circle (2pt);
\node[below] at (3,0) {$\beta_4$};   \fill (3,0) circle (2pt);

\draw(0,1) edge (0,0);
\draw(1,1) edge (1,0);
\draw(2,1) edge (2,0);
\draw(3,1) edge (3,0);
\end{tikzpicture}
\hspace{1cm}
\begin{tikzpicture}
\node[above] at (0,1) {$\alpha_1$};   \fill (0,1) circle (2pt);
\node[above] at (1,1) {$\alpha_2$};   \fill (1,1) circle (2pt);
\node[above] at (2,1) {$\alpha_3$};   \fill (2,1) circle (2pt);
\node[above] at (3,1) {$\alpha_4$};   \fill (3,1) circle (2pt);

\node[below] at (0,0) {$\beta_1$};  \fill (0,0) circle (2pt);
\node[below] at (1,0) {$\beta_2$};   \fill  (1,0) circle (2pt);
\node[below] at (2,0) {$\beta_3$};   \fill  (2,0) circle (2pt);
\node[below] at (3,0) {$\beta_4$};   \fill (3,0) circle (2pt);

\draw(0,1) edge (1,0);
\draw(1,1) edge (2,0);
\draw(2,1) edge (3,0);
\draw(3,1) edge (0,0);
\end{tikzpicture}

\vspace{0.5cm}

\begin{tikzpicture}
\node[above] at (0,1) {$\alpha_1$};   \fill (0,1) circle (2pt);
\node[above] at (1,1) {$\alpha_2$};   \fill (1,1) circle (2pt);
\node[above] at (2,1) {$\alpha_3$};   \fill (2,1) circle (2pt);
\node[above] at (3,1) {$\alpha_4$};   \fill (3,1) circle (2pt);

\node[below] at (0,0) {$\beta_1$};  \fill (0,0) circle (2pt);
\node[below] at (1,0) {$\beta_2$};   \fill  (1,0) circle (2pt);
\node[below] at (2,0) {$\beta_3$};   \fill  (2,0) circle (2pt);
\node[below] at (3,0) {$\beta_4$};   \fill (3,0) circle (2pt);

\draw(0,1) edge (2,0);
\draw(1,1) edge (3,0);
\draw(2,1) edge (0,0);
\draw(3,1) edge (1,0);
\end{tikzpicture}
\hspace{1cm}
\begin{tikzpicture}
\node[above] at (0,1) {$\alpha_1$};   \fill (0,1) circle (2pt);
\node[above] at (1,1) {$\alpha_2$};   \fill (1,1) circle (2pt);
\node[above] at (2,1) {$\alpha_3$};   \fill (2,1) circle (2pt);
\node[above] at (3,1) {$\alpha_4$};   \fill (3,1) circle (2pt);

\node[below] at (0,0) {$\beta_1$};  \fill (0,0) circle (2pt);
\node[below] at (1,0) {$\beta_2$};   \fill  (1,0) circle (2pt);
\node[below] at (2,0) {$\beta_3$};   \fill  (2,0) circle (2pt);
\node[below] at (3,0) {$\beta_4$};   \fill (3,0) circle (2pt);

\draw(0,1) edge (3,0);
\draw(1,1) edge (0,0);
\draw(2,1) edge (1,0);
\draw(3,1) edge (2,0);
\end{tikzpicture}

\end{center}

\bigskip

\subsection{The Case $\lambda = (a^b)$} \label{formulas} Suppose now that our permutation $g$ has $b$ cycles of length $a$, where $ab = 2m$. Recall that we can conjugate $g$ by a permutation $\sigma$ by applying $\sigma$ to each symbol of each cycle of $g$. That is, if $g = (\alpha_1 \ \alpha_2 \ \cdots \ \alpha_a)(\beta_1 \ \beta_2 \ \cdots \ \beta_a) \cdots$, we have
\[ \sigma g \sigma^{-1} = (\sigma(\alpha_1) \ \sigma(\alpha_2) \ \cdots \ \sigma(\alpha_a))(\sigma(\beta_1) \ \sigma(\beta_2) \ \cdots \ \sigma(\beta_a))\cdots \]
It follows that if $\sigma$ commutes with $g$, conjugation by $\sigma$ sends one cycle of $g$ to another cycle of $g$. Thus we can take pairs of cycles of $g$ and look for the matchings on $2a$ numbers that commute with the product of these pairs following Lemma~\ref{brick_matchings}. If we have an odd number of cycles, the unpaired cycle will commute with a power of itself, that power being $a/2$. The product of these will be a matching which commutes with $g$ itself. 

To illustrate, we'll quickly compute a matching which commutes with the permutation $g = (1 \ 2 \ 3 \ 4)(5 \ 6 \ 7 \ 8)(9 \ 10 \ 11 \ 12)$. Suppose we first pair the first and third cycle, and look for matchings which commute with $(1 \ 2 \ 3 \ 4)(9 \ 10 \ 11 \ 12)$. By Lemma~\ref{brick_matchings}, we have five to choose from, such as $(1 \ 10)(2 \ 11)(3 \ 12)(4 \ 9)$. We still have one cycle of $g$ left, and there is only one matching to commute with it: $(5 \ 7)(6 \ 8)$. Hence, a matching that commutes with $g$ is $(1 \ 10)(2 \ 11)(3 \ 12)(4 \ 9)(5 \ 7)(6 \ 8)$.  

Once again, we turn to diagrams to simplify. After fixing a cycle expression for $g$, we draw a diagram with $b$ rows of $a$ nodes, and label the nodes of the $i$th row with the entries of the $i$th cycle of $g$, as before. Again, we can define a matching by drawing an edge between pairs of nodes. The following diagrams illustrate the case where $g$ has cycle type $\lambda = (4^5)$. The diagram on the left displays a matching, but not a matching that will commute with the permutation $g$, as it fails to map one cycle of $g$ to another. The center diagram meets this requirement, but does not satisfy Lemma~\ref{brick_matchings}. Finally, the diagram on the right corresponds to a matching that will commute with $g$. For simplicity, we have suppressed labeling of the nodes. 

\bigskip

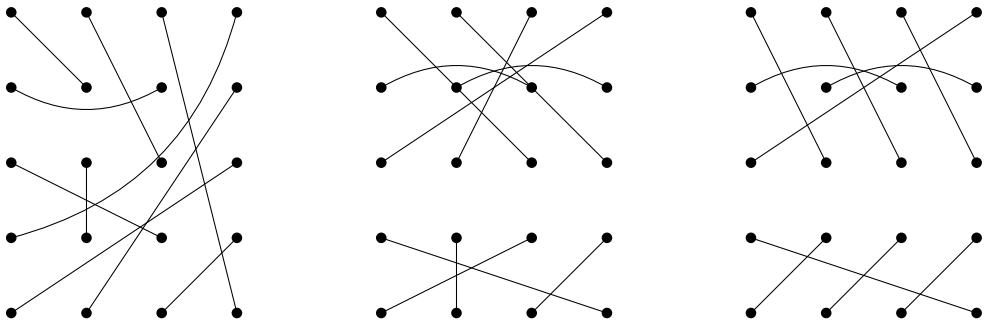
\begin{center}
\begin{tikzpicture}
\fill (0,0) circle (2pt);
\fill (1,0) circle (2pt);
\fill (2,0) circle (2pt);
 \fill (3,0) circle (2pt);

\fill (0,1) circle (2pt);
\fill (1,1) circle (2pt);
\fill (2,1) circle (2pt);
\fill (3,1) circle (2pt);

\fill (0,2) circle (2pt);
\fill (1,2) circle (2pt);
\fill (2,2) circle (2pt);
\fill (3,2) circle (2pt);

\fill (0,3) circle (2pt);
\fill (1,3) circle (2pt);
\fill (2,3) circle (2pt);
\fill (3,3) circle (2pt);

\fill (0,4) circle (2pt);
\fill (1,4) circle (2pt);
\fill (2,4) circle (2pt);
\fill (3,4) circle (2pt);

\draw(0,4) edge (1,3);
\draw(1,4) edge (2,2);
\draw(2,4) edge (3,0);
\draw(3,4) edge[bend left] (0,1);
\draw(0,3) edge[bend right] (2,3);
\draw(2,1) edge (0,2);
\draw(3,2) edge (0,0);
\draw(1,0) edge (3,3);
\draw(3,1) edge (2,0);
\draw(1,1) edge (1,2);

\end{tikzpicture}
\hspace{1.5cm}
\begin{tikzpicture}
\fill (0,0) circle (2pt);
\fill (1,0) circle (2pt);
\fill (2,0) circle (2pt);
 \fill (3,0) circle (2pt);

\fill (0,1) circle (2pt);
\fill (1,1) circle (2pt);
\fill (2,1) circle (2pt);
\fill (3,1) circle (2pt);

\fill (0,2) circle (2pt);
\fill (1,2) circle (2pt);
\fill (2,2) circle (2pt);
\fill (3,2) circle (2pt);

\fill (0,3) circle (2pt);
\fill (1,3) circle (2pt);
\fill (2,3) circle (2pt);
\fill (3,3) circle (2pt);

\fill (0,4) circle (2pt);
\fill (1,4) circle (2pt);
\fill (2,4) circle (2pt);
\fill (3,4) circle (2pt);

\draw(0,4) edge (2,2);
\draw(1,4) edge (3,2);
\draw(2,4) edge (1,2);
\draw(3,4) edge (0,2);

\draw(0,3) edge[bend left] (2,3);
\draw(1,3) edge[bend left] (3,3);

\draw(0,1) edge (3,0);
\draw(1,1) edge (1,0);
\draw(2,1) edge (0,0);
\draw(3,1) edge (2,0);
\end{tikzpicture}
\hspace{1.5cm}
\begin{tikzpicture}
\fill (0,0) circle (2pt);
\fill (1,0) circle (2pt);
\fill (2,0) circle (2pt);
\fill (3,0) circle (2pt);

\fill (0,1) circle (2pt);
\fill (1,1) circle (2pt);
\fill (2,1) circle (2pt);
\fill (3,1) circle (2pt);

\fill (0,2) circle (2pt);
\fill (1,2) circle (2pt);
\fill (2,2) circle (2pt);
\fill (3,2) circle (2pt);

\fill (0,3) circle (2pt);
\fill (1,3) circle (2pt);
\fill (2,3) circle (2pt);
\fill (3,3) circle (2pt);

\fill (0,4) circle (2pt);
\fill (1,4) circle (2pt);
\fill (2,4) circle (2pt);
\fill (3,4) circle (2pt);

\draw(0,4) edge (1,2);
\draw(1,4) edge (2,2);
\draw(2,4) edge (3,2);
\draw(3,4) edge (0,2);

\draw(0,3) edge[bend left] (2,3);
\draw(1,3) edge[bend left] (3,3);

\draw(0,1) edge (3,0);
\draw(1,1) edge (0,0);
\draw(2,1) edge (1,0);
\draw(3,1) edge (2,0);
\end{tikzpicture}
\end{center}

\bigskip

We have now established a convenient way of establishing which matchings commute with a permutation of shape $(a^b)$, it is simply a matter of counting them. To do this, we will compress the diagrams which correspond to an eligible matching as follows: for each cycle of the permutation, draw a single node. By Proposition~\ref{brick_matchings}, there are $a$ matchings that will commute with two distinct cycles. Assign a color to each of these choices, and color the edge between the corresponding nodes accordingly. If matching contains the $a/2$-th power of a cycle, the corresponding node is left isolated. Thus $N((a^b))$ is obtained by counting graphs of the form
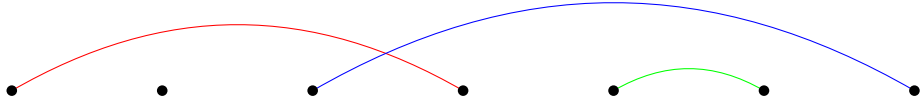
\begin{center}
\begin{tikzpicture}
\draw(0,0) edge[bend left, red] (6,0);
\draw(4,0) edge[bend left, blue] (12,0);
\draw(8,0) edge[bend left, green] (10,0);
\fill (0,0) circle (2pt);
\fill (2,0) circle (2pt);
\fill (4,0) circle (2pt);
\fill (6,0) circle (2pt);
\fill (8,0) circle (2pt);
\fill (10,0) circle (2pt);
\fill (12,0) circle (2pt);
\end{tikzpicture}
\end{center}

\bigskip

We consider three sub cases, depending on the parity of $a$ and $b$ (recall we cannot allow both $a$ and $b$ to be odd, as we require $ab$ to be even):

\subsubsection{$a$ odd, $b$ even} If $a$ is odd, any matching that commutes with $g$ cannot contain a power of a cycle of $g$, so we have no isolated nodes. The number of uncolored diagrams in this case is equal to the number of matchings on $b$ letters, and each of the $b/2$ edges of the diagram can be colored in one of $a$ ways. Hence we have
\[ N((a^b)) = \frac{b!a^{b/2}}{2^{b/2}\left(\frac{b}{2}\right)!} \]

\subsubsection{$a$ even, $b$ even} If $a$ and $b$ are even, we can allow an even number of isolated nodes. If our diagram contains $i$ isolated nodes, the number of uncolored diagrams is equal to the number of permutations on $b$ with shape $(1^i2^{b-i})$, and each of the $(b-i)/2$ edges of the diagram can be colored in one of $a$ ways. Hence we have
\[ N((a^b)) = \sum_{i=0, i \mbox{ \tiny even}}^b \frac{b!a^{(b-i)/2}}{i!(\frac{b-i}{2})! 2^{(b-i)/2}}  \]

\subsubsection{$a$ even, $b$ odd} Similarly, if $b$ is odd we can allow an odd number of isolated nodes, and so
\[ N((a^b)) = \sum_{i=1, i \mbox{ \tiny odd}}^b \frac{b!a^{(b-i)/2}}{i!(\frac{b-i}{2})! 2^{(b-i)/2}}  \]

\bigskip

\subsection{The General Case $\lambda = (1^{b_1}2^{b_2} \cdots t^{b_t})$}

\begin{lemma}\label{length_matchings}
Let $g$ be a permutation and and let $\sigma$ be a matching such that $\sigma g \sigma = g$. If $(\alpha \ \beta)$ appears in $\sigma$, then $\alpha, \beta \in \{1, 2, \ldots, n\}$ appear in cycles of the same length in $g$. 
\end{lemma} 

\begin{proof}
Suppose $\alpha$ appears in the cycle $(\alpha \ \alpha_1 \ \alpha_2 \ \cdots \ \alpha_k)$ and $\beta$ appears in the cycle $(\beta \ \beta_1 \ \beta_2 \ \cdots \ \beta_l)$. Assume without loss of generality that $k \leq l$. We have
\[ \alpha_1 = g(\alpha) = \sigma g \sigma (\alpha) = \sigma g(\beta) = \sigma(\beta_1) \]
and so $(\alpha_1 \ \beta_1)$ is a transposition appearing in $\sigma$. Similarly,
\[ \alpha_2 = g(\alpha_1) = \sigma g \sigma (\alpha_1) = \sigma g(\beta_1) = \sigma(\beta_2) \]
and so  $(\alpha_2 \ \beta_2)$ is also a transposition appearing in $\sigma$. Continuing in this way, we find that $\sigma$ must contain the transpositions
\[ (\alpha \ \beta), (\alpha_1 \ \beta_1), (\alpha_2 \ \beta_2), \ldots, (\alpha_k \ \beta_k) \]
Suppose for contradiction that $k<l$. Then there must exist some $\gamma \in \{1, 2, \ldots, n\}$ such that $\gamma \notin \{\alpha, \alpha_1, \ldots , \alpha_k\}$ and $(\gamma \ \beta_l)$ is a transposition appearing in $\sigma$. Then
\[ \beta = g(\beta_l) = \sigma g \sigma(\beta_l) = \sigma g(\gamma) \]
By hypothesis $\beta = \sigma (\alpha)$, so $\alpha = g(\gamma)$. But $\alpha = g(\alpha_k)$, a contradiction, since $\gamma \neq \alpha_k$. Hence $k = l$.
\end{proof}

Thus, if $g$ has $b_i$ cycles of length $i$, we need only count the number of matchings which commute with the product of these $b_i$ cycles, which we can calculate using the formulas already determined. To find the total number of matchings which commute with $g$, we simply take the product. That is,
\[ N(\lambda) = N((1^{b_1}))N((2^{b_2})) \cdots N((t^{b_t})) \]

As an example, consider the permutation 
\[ g = (1 \ 14 \ 7 \ 11)(5 \ 16 \ 12)(6 \ 20 \ 9)(8 \ 17 \ 13)(10 \ 2 \ 15 \ 4)(18 \ 3 \ 19)\]
We note that $g$ has two cycles of length four, and by Lemma~\ref{brick_matchings}, we have
\[ N((4^2)) = 4 + 1 = 5 \]
In addition, $g$ has four cycles of length three, and
\[ N((3^4)) = \frac{4!3^{4/2}}{2^{4/2}(4/2)!} = 27 \]
So the number of matchings that will commute with $g$ is $5 \cdot 27 = 135$. 

\subsection{Data}

Table 1 displays dimension of $\scP^{2m}(\V)^{\K}$ for several values of $r$ and $m$ generated by the formulas of Theorems~\ref{main_theorem} and~\ref{Nlambda_theorem}.

\bigskip

\begin{table}[H]\label{stabletable}
\caption{{\tiny $\dim \scP^{2m}(\V)^{\K}$ where $\K = \prod_{i=1}^r O({n_i})$, $\V = \prod_{i=1}^r \C^{n_i}$, and $n_i \geq 2m$ for $1 \leq i \leq r$}}
\begin{tabular}{c||llllll}
 $r \backslash m$&1 & 2& 3 & 4 & 5 & 6 \\
\hline\hline
1 & 1 & 1 & 1 & 1 & 1 & 1 \\
2 & 1 & 2 & 3 & 5 & 7 & 11 \\
3 & 1 & 5 & 16 & 86 & 448 & 3580 \\
4 & 1 & 14 & 132 & 4154 & 234004 & 24791668 \\
5 & 1 & 41 & 1439 & 343101 & 208796298 & 253588562985 \\
6 & 1 & 122 & 18373 & 33884745 & 196350215004 &
2634094790963313 \\
7 & 1 & 365 & 254766 & 3505881766 & 185471778824111 &
27380169200102651288 \\
8 & 1 & 1094 & 3680582 & 366831842914 & 175264150734326927
& 284615877731708760168866 
\end{tabular}
\end{table}
 
\bigskip

\section{The Invariants and a Graph Interpretation}\label{graph_section}

We now seek a description of the invariants themselves. We begin by setting up notation for an arbitrary tensor in $\V$. Let $e_{i} \in \C^n$ denote the vector with a 1 in row $i$ and 0 elsewhere. Note that $(e_1, e_2, \ldots, e_n)$ is an ordered basis for $\C^n$. An arbitrary tensor in $\V$ is of the form
\[ \sum x_{i_1i_2\cdots i_r}e_{i_1}\otimes e_{i_2}\otimes \cdots \otimes e_{i_r} \]
where $i_j$ ranges from $1$ to $n_j$ and $x_{i_1i_2\cdots i_r}$ is a complex scalar. 

Recall that $\tI^r$ denotes the set of $r$-tuples of matchings on $2m$ letters. The group $S_{2m}$ acts on this set by simultaneous conjugation. We choose a representative $(\tau_1, \tau_2, \ldots, \tau_r)$ from each orbit under this action, which we denote by
\[ [\tau_1, \tau_2, \ldots, \tau_r] = \{\sigma(\tau_1, \tau_2, \ldots, \tau_r)\sigma^{-1} \ | \ \sigma \in S_{2m} \} \]
Fix an ordering on the cycles of each $\tau_i$, and let $j_i^k$ denote the cycle containing $k$ in $\tau_i$. For instance, if $\tau_2 = (1 \ 3)(2 \ 4)$, then $j_2^4 = 2$, since 4 appears in the second cycle of $\tau_2$. The invariant associated to $[\tau_1, \tau_2, \ldots, \tau_r]$ can now written as the sum of terms of the form
\[ x_{a^{(1)}_{j_1^1} a^{(2)}_{j_2^1} \cdots a^{(r )}_{j_r^1}}x_{a^{(1)}_{j_1^2} a^{(2)}_{j_2^2} \cdots a^{(r )}_{j_r^2}} \cdots x_{a^{(1)}_{j_1^{2m}} a^{(2)}_{j_2^{2m}} \cdots a^{(r )}_{j_r^{2m}}}   \]
where each $a_i^{(t)}$ ranges from 1 to $n_k$.

A $k$-regular graph on $n$ vertices is a graph in which each of the $n$ vertices has degree $k$; that is, each vertex is met by exactly $k$ edges. An edge coloring of a graph is a coloring of these edges so that no two adjacent edges are the same color (this is also referred to as a 1-factorization of the graph). In this case of a $k$-regular graph, an edge coloring implies that each vertex is met by $k$ edges, which are each a distinct color. 

We now present a bijection between the orbits of the action of $S_{2m}$ on $\tI^r$ and the isomorphism classes of edge colored $r$-regular graphs with $2m$ vertices. To construct the graph associated to $[\tau_1, \tau_2, \ldots, \tau_r]$, number the vertices of the graph from 1 to $2m$. An edge of color $i$ is drawn between vertices $j$ and $k$ if $\tau_i$ contains the cycle $(j \ k)$. Repeating this process for all $r$ matchings, we obtain an undirected graph with $2m$ vertices and $r$ colors. 

To illustrate, we consider the case $r = 3$, $m=2$ where we have chosen the 3-tuple of matchings
\[ (\tau_1, \tau_2, \tau_3) = ((1 \ 2)(3 \ 4), (1 \ 3)(2 \ 4), (1 \ 4)(2 \ 3)) \]
The invariant associated to $[\tau_1,\tau_2,\tau_3]$ is then
\[ \sum_{a_1^{(1)}, a_2^{(1)}}^{n_1} \sum_{a_1^{(2)},a_2^{(2)}}^{n_2} \sum_{a_1^{(3)},a_2^{(3)}}^{n_3} x_{a_1^{(1)}a_1^{(2)}a_1^{(3)}}x_{a_1^{(1)}a_2^{(2)}a_2^{(3)}}x_{a_2^{(1)}a_1^{(2)}a_2^{(3)}}x_{a_2^{(1)}a_2^{(2)}a_1^{(3)}} \]
To construct the graph associated to $[\tau_1,\tau_2,\tau_3]$, we encode $\tau_1, \tau_2, \tau_3$ with the colors black, red, and blue, respectively. Forgetting the labels of the vertices leaves us with a representative of the isomorphism class containing the graph.

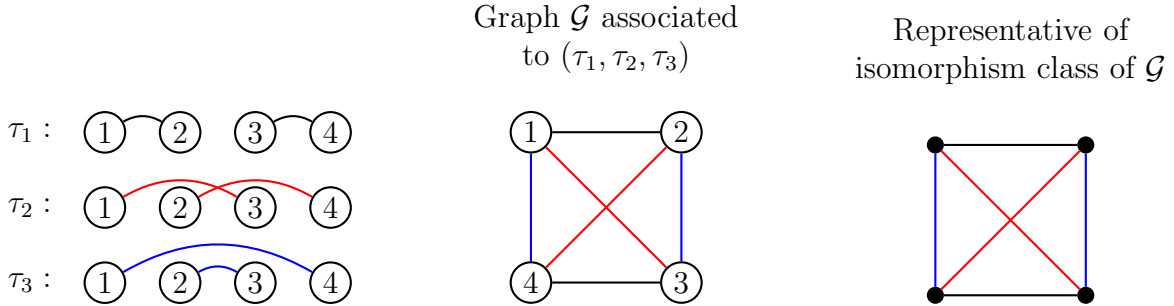
\begin{center}
\begin{tikzpicture}
      	\node (1) at (0,0) [circle,draw, thick,inner sep=2pt] {1};
      	\node (2) at (1,0) [circle,draw, thick,inner sep=2pt] {2};
	\node (3) at (2,0) [circle,draw, thick,inner sep=2pt] {3};
	\node (4) at (3,0) [circle,draw, thick,inner sep=2pt] {4};
	\draw[thick, bend left] (1) to (2);
	\draw[thick, bend left] (3) to (4);
	\node (5) at (0,-1) [circle,draw, thick,inner sep=2pt] {1};
      	\node (6) at (1,-1) [circle,draw, thick,inner sep=2pt] {2};
	\node (7) at (2,-1) [circle,draw, thick,inner sep=2pt] {3};
	\node (8) at (3,-1) [circle,draw, thick,inner sep=2pt] {4};
	\draw[thick, bend left,red] (5) to (7);
	\draw[thick, bend left,red] (6) to (8);
	\node (9) at (0,-2) [circle,draw, thick,inner sep=2pt] {1};
      	\node (10) at (1,-2) [circle,draw, thick,inner sep=2pt] {2};
	\node (11) at (2,-2) [circle,draw, thick,inner sep=2pt] {3};
	\node (12) at (3,-2) [circle,draw, thick,inner sep=2pt] {4};
	\draw[thick, bend left,blue] (10) to (11);
	\draw[thick, bend left,blue] (9) to (12);
	\node (t12) at (-1,0) {$\tau_1:$};
	\node (t22) at (-1,-1) {$\tau_2:$};
	\node (t32) at (-1,-2) {$\tau_3:$};
\end{tikzpicture}
\hspace{1.2cm}
\begin{tikzpicture}
      	\node (1) at (0,2) [circle,draw, thick,inner sep=2pt] {1};
      	\node (2) at (2,2) [circle,draw, thick,inner sep=2pt] {2};
	\node (3) at (2,0) [circle,draw, thick,inner sep=2pt] {3};
	\node (4) at (0,0) [circle,draw, thick,inner sep=2pt] {4};
	\draw[thick] (1) to (2); \draw[thick] (3) to (4);
	\draw[thick, red] (1) to (3); \draw[thick, red] (2) to (4);
	\draw[thick, blue] (1) to (4); \draw[thick, blue] (2) to (3);
	\node (t12) at (1,3.5) {Graph $\mathcal{G}$ associated};
	\node (t22) at (1,3) {to $(\tau_1, \tau_2, \tau_3)$};
\end{tikzpicture}
\hspace{1cm}
\begin{tikzpicture}
      	\node (1) at (0,2) [fill,circle,draw, thick,inner sep=2pt] {};
      	\node (2) at (2,2) [fill,circle,draw, thick,inner sep=2pt] {};
	\node (3) at (2,0) [fill,circle,draw, thick,inner sep=2pt] {};
	\node (4) at (0,0) [fill,circle,draw, thick,inner sep=2pt] {};
	\draw[thick] (1) to (2); \draw[thick] (3) to (4);
	\draw[thick, red] (1) to (3); \draw[thick, red] (2) to (4);
	\draw[thick, blue] (1) to (4); \draw[thick, blue] (2) to (3);
	\node (t12) at (1,3.5) {Representative of };
	\node (t22) at (1,3) {isomorphism class of $\mathcal{G}$};
\end{tikzpicture}
\end{center}
\bigskip
Note that any graph isomorphic to this graph will correspond to a 3-tuple of matchings that belongs to the same orbit as our original choice. 

A list of representatives of all isomorphism classes of 3-regular graphs on four vertices is shown in Table 2, along with the corresponding invariant. We have chosen a representative $(\tau_1, \tau_2,\tau_3)$ of each orbit so that $\tau_1 = (1 \ 2)(3 \ 4)$.

Finally, we show a way of encoding these invariants with forests of phylogenetic trees. As explored in the introduction, Diaconis and Holmes \cite{Diaconis} provide a bijection between matchings on $2m$ letters and phylogenetic trees with $m+1$ labelled leaves. When given an $r$-tuple of matchings, we can define a forest of $r$ trees with roots labelled $1, 2, \ldots, r$. Note that we consider two forests to be equivalent if the individual trees within the forest are equivalent in the traditional way. As an example, we again consider the 3-tuple 
\[ (\tau_1, \tau_2, \tau_3) = ((1 \ 2)(3 \ 4), (1 \ 3)(2 \ 4), (1 \ 4)(2 \ 3)) \]
which can be interpreted as the forest
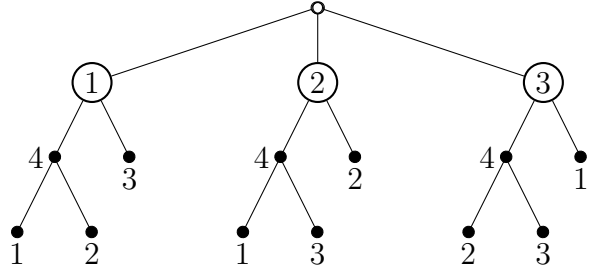
\begin{center}
\begin{tikzpicture}
      	\node (1) at (0.5,3) [circle,draw, thick,inner sep=1.8pt] {1};
      	\node[below] at (-0.5,1) {1};
	\node[below] at (1,2) {3};
	\node[below] at (0.5,1) {2};
	\node[left] at (0,2)  {4};
	\draw[-*, shorten >=-2] (1) to (1,2);
	\draw[-*, shorten >=-2] (1) to (0,2);
	\draw[-*, shorten >=-2] (0,2) to (-0.5,1);
	\draw[-*, shorten >=-2] (0,2) to (0.5,1);
	
	\node (2) at (3.5,3) [circle,draw, thick,inner sep=1.8pt] {2};
      	\node[below] at (2.5,1) {1};
	\node[below] at (3.5,1) {3};
	\node[below] at (4,2) {2};
	\node[left] at (3,2)  {4};
	\draw[-*, shorten >=-2] (2) to (4,2);
	\draw[-*, shorten >=-2] (2) to (3,2);
	\draw[-*, shorten >=-2] (3,2) to (3.5,1);
	\draw[-*, shorten >=-2] (3,2) to (2.5,1);
	
	\node (3) at (6.5,3) [circle,draw, thick,inner sep=1.8pt] {3};
      	\node[below] at (5.5,1) {2};
	\node[left] at (6,2) {4};
	\node[below] at (6.5,1) {3};
	\node[below] at (7,2)  {1};
	\draw[-*, shorten >=-2] (3) to (7,2);
	\draw[-*, shorten >=-2] (3) to (6,2);
	\draw[-*, shorten >=-2] (6,2) to (6.5,1);
	\draw[-*, shorten >=-2] (6,2) to (5.5,1);
	
	\draw[-o, shorten >=-2] (1) to (3.5,4);
	\draw[-o, shorten >=-2] (2) to (3.5,4);
	\draw[-o, shorten >=-2] (3) to (3.5,4);
\end{tikzpicture}
\end{center}
after we have labelled all vertices following the method outlined in section 1.

The invariant associated to a forest with $r$ trees can again be written as a sum of terms of the form
\[ x_{a^{(1)}_{j_1^1} a^{(2)}_{j_2^1} \cdots a^{(r )}_{j_r^1}}x_{a^{(1)}_{j_1^2} a^{(2)}_{j_2^2} \cdots a^{(r )}_{j_r^2}} \cdots x_{a^{(1)}_{j_1^{2m}} a^{(2)}_{j_2^{2m}} \cdots a^{(r )}_{j_r^{2m}}}   \]
where the index $(i)$ of the subscript refers to the root of an individual tree in the forest, and $j_i^s = j_i^t$ if $s$ and $t$ are siblings in the $i$th tree. To build the invariant associated with our particular forest, we begin by examining the tree labelled (1). This tree has two pairs of siblings, $(1,2)$ and $(3,4)$, and so our invariant begins as
\[ x_{a^{(1)}_1}\_\_ \ x_{a^{(1)}_1}\_\_ \ x_{a^{(1)}_2}\_\_ \ x_{a^{(1)}_2}\_\_ \] 
Next, we see that the second tree has sibling pairs $(1,3)$ and $(2,4)$, so we continue building the invariant:
\[ x_{a^{(1)}_1a^{(2)}_1}\_ \ x_{a^{(1)}_1a^{(2)}_2}\_ \ x_{a^{(1)}_2a^{(2)}_1}\_ \ x_{a^{(1)}_2a^{(2)}_2}\_ \] 
Finally, the last placeholders of the invariant are filled by observing that the third tree has sibling pairs $(2,3)$ and $(1,4)$. The result is identical to the invariant associated to the same tuple of matchings determined earlier.
\[ x_{a^{(1)}_1a^{(2)}_1a^{(3)}_1} x_{a^{(1)}_1a^{(2)}_2a^{(3)}_2} x_{a^{(1)}_2a^{(2)}_1a^{(3)}_2}  x_{a^{(1)}_2a^{(2)}_2a^{(3)}_1} \]

\bigskip

\begin{table}[H]\label{invariants}
\caption{Representative of graph isomorphism class and corresponding invariant in the case $r=3,m=2$}
\begin{tabular}{c}   
\begin{tikzpicture}[scale=1.5]
\draw[thick] (1.3,-0.2) to (1.3,8.3);
\draw[thick] (1.3,-0.2) to (1.3,8.3);
\draw[thick, dashed] (-0.3,1.375) to (9.7,1.375);
\draw[thick, dashed] (-0.3,3.125) to (9.7,3.125);
\draw[thick, dashed] (-0.3,4.875) to (9.7,4.875);
\draw[thick, dashed] (-0.3,6.625) to (9.7,6.625);

\node (11) at (0,8) [fill,circle,inner sep=2pt]{};
\node (12) at (1,8) [fill,circle,inner sep=2pt]{};
\node (13) at (0,7) [fill,circle,inner sep=2pt]{};
\node (14) at (1,7) [fill,circle,inner sep=2pt]{};
\draw[very thick] (11) to (12); 
\draw[very thick] (13) to (14); 
\draw[very thick,red] (11) to [bend left=40] (12); 
\draw[very thick,red] (13) to [bend left=40] (14); 
\draw[very thick,blue] (11) to [bend right=40] (12); 
\draw[very thick,blue] (13) to [bend right=40] (14); 
\node at (1.5,7.5) [right] {$\sum_{a_1^{(1)}, a_2^{(1)}}^{n_1} \sum_{a_1^{(2)},a_2^{(2)}}^{n_2} \sum_{a_1^{(3)},a_2^{(3)}}^{n_3} x_{a_1^{(1)}a_1^{(2)}a_1^{(3)}}x_{a_1^{(1)}a_1^{(2)}a_1^{(3)}}x_{a_2^{(1)}a_2^{(2)}a_2^{(3)}}x_{a_2^{(1)}a_2^{(2)}a_2^{(3)}}$};
  
\node (21) at (0,6.25) [fill,circle,inner sep=2pt]{};
\node (22) at (1,6.25) [fill,circle,inner sep=2pt]{};
\node (23) at (0,5.25) [fill,circle,inner sep=2pt]{};
\node (24) at (1,5.25) [fill,circle,inner sep=2pt]{};
\draw[very thick] (21) to (22); 
\draw[very thick] (23) to (24); 
\draw[very thick,red] (21) to [bend left=40] (22); 
\draw[very thick,red] (24) to [bend left=40] (23); 
\draw[very thick,blue] (21) to (23); 
\draw[very thick,blue] (22) to (24); 
\node at (1.5,5.75) [right]  {$\sum_{a_1^{(1)}, a_2^{(1)}}^{n_1} \sum_{a_1^{(2)},a_2^{(2)}}^{n_2} \sum_{a_1^{(3)},a_2^{(3)}}^{n_3} x_{a_1^{(1)}a_1^{(2)}a_1^{(3)}}x_{a_1^{(1)}a_1^{(2)}a_2^{(3)}}x_{a_2^{(1)}a_2^{(2)}a_2^{(3)}}x_{a_2^{(1)}a_2^{(2)}a_1^{(3)}}$ };
     
\node (31) at (0,4.5) [fill,circle,inner sep=2pt]{};
\node (32) at (1,4.5) [fill,circle,inner sep=2pt]{};
\node (33) at (0,3.5) [fill,circle,inner sep=2pt]{};
\node (34) at (1,3.5) [fill,circle,inner sep=2pt]{};
\draw[very thick] (31) to (32); 
\draw[very thick] (33) to (34); 
\draw[very thick,blue] (31) to [bend left=40] (32); 
\draw[very thick,blue] (34) to [bend left=40] (33); 
\draw[very thick,red] (31) to (33); 
\draw[very thick,red] (32) to (34); 
\node at (1.5,4) [right] {$\sum_{a_1^{(1)}, a_2^{(1)}}^{n_1} \sum_{a_1^{(2)},a_2^{(2)}}^{n_2} \sum_{a_1^{(3)},a_2^{(3)}}^{n_3} x_{a_1^{(1)}a_1^{(2)}a_1^{(3)}}x_{a_1^{(1)}a_2^{(2)}a_1^{(3)}}x_{a_2^{(1)}a_2^{(2)}a_2^{(3)}}x_{a_2^{(1)}a_1^{(2)}a_2^{(3)}}$ };

\node (41) at (0,2.75) [fill,circle,inner sep=2pt]{};
\node (42) at (1,2.75) [fill,circle,inner sep=2pt]{};
\node (43) at (0,1.75) [fill,circle,inner sep=2pt]{};
\node (44) at (1,1.75) [fill,circle,inner sep=2pt]{};
\draw[very thick] (41) to (42); 
\draw[very thick] (43) to (44); 
\draw[very thick,red] (41) to [bend left=30] (43); 
\draw[very thick,red] (42) to [bend left=30] (44); 
\draw[very thick,blue] (41) to [bend right=30] (43); 
\draw[very thick,blue] (42) to [bend right=30] (44); 
\node at (1.5,2.25) [right] {$\sum_{a_1^{(1)}, a_2^{(1)}}^{n_1} \sum_{a_1^{(2)},a_2^{(2)}}^{n_2} \sum_{a_1^{(3)},a_2^{(3)}}^{n_3} x_{a_1^{(1)}a_1^{(2)}a_1^{(3)}}x_{a_1^{(1)}a_2^{(2)}a_2^{(3)}}x_{a_2^{(1)}a_2^{(2)}a_2^{(3)}}x_{a_2^{(1)}a_1^{(2)}a_1^{(3)}}$ };

\node (51) at (0,1) [fill,circle,inner sep=2pt]{};
\node (52) at (1,1) [fill,circle,inner sep=2pt]{};
\node (53) at (0,0) [fill,circle,inner sep=2pt]{};
\node (54) at (1,0) [fill,circle,inner sep=2pt]{};
\draw[very thick] (51) to (52); 
\draw[very thick] (53) to (54); 
\draw[very thick,blue] (51) to  (53); 
\draw[very thick,blue] (52) to (54); 
\draw[very thick,red] (51) to (54); 
\draw[very thick,red] (52) to (53); 
\node at (1.5,0.5) [right] {$ \sum_{a_1^{(1)}, a_2^{(1)}}^{n_1} \sum_{a_1^{(2)},a_2^{(2)}}^{n_2} \sum_{a_1^{(3)},a_2^{(3)}}^{n_3} x_{a_1^{(1)}a_1^{(2)}a_1^{(3)}}x_{a_1^{(1)}a_2^{(2)}a_2^{(3)}}x_{a_2^{(1)}a_1^{(2)}a_2^{(3)}}x_{a_2^{(1)}a_2^{(2)}a_1^{(3)}} $};
        \end{tikzpicture} 
\end{tabular}
\end{table}


\begin{bibdiv}
\bibliographystyle{alphabetic}

\begin{biblist}

\bib{Brauer}{article}{
    AUTHOR = {Brauer, Richard},
     TITLE = {On algebras which are connected with the semisimple continuous
              groups},
   JOURNAL = {Ann. of Math. (2)},
  FJOURNAL = {Annals of Mathematics. Second Series},
    VOLUME = {38},
      YEAR = {1937},
    NUMBER = {4},
     PAGES = {857--872},
      ISSN = {0003-486X},
     CODEN = {ANMAAH},
   MRCLASS = {Contributed Item},
  MRNUMBER = {1503378},

}

\bib{Diaconis}{article}{
    AUTHOR = {Diaconis, Persi W.},
	AUTHOR ={ Holmes, Susan P.},
     TITLE = {Matchings and phylogenetic trees},
   JOURNAL = {Proc. Natl. Acad. Sci. USA},
  FJOURNAL = {Proceedings of the National Academy of Sciences of the United
              States of America},
    VOLUME = {95},
      YEAR = {1998},
    NUMBER = {25},
     PAGES = {14600--14602 (electronic)},
      ISSN = {1091-6490},
     CODEN = {PNASFB},
   MRCLASS = {92D15},
  MRNUMBER = {1665632},
}


\bib{GW}{book} {
    	AUTHOR = {Goodman, Roe},
	AUTHOR = {Wallach, Nolan R.},
     	TITLE = {Symmetry, representations, and invariants},
    	SERIES = {Graduate Texts in Mathematics},
    	VOLUME = {255},
 	PUBLISHER = {Springer},
   	ADDRESS = {Dordrecht},
      	YEAR = {2009},
	REVIEW={\MR{2522486 (2011a:20119)}}
}

\bib{Stanley}{article}{
   AUTHOR = {Hanlon, Philip J.},
Author={ Stanley, Richard P. },
Author ={ Stembridge, John
              R.},
     TITLE = {Some combinatorial aspects of the spectra of normally
              distributed random matrices},
 BOOKTITLE = {Hypergeometric functions on domains of positivity, {J}ack
              polynomials, and applications ({T}ampa, {FL}, 1991)},
    SERIES = {Contemp. Math.},
    VOLUME = {138},
     PAGES = {151--174},
 PUBLISHER = {Amer. Math. Soc.},
   ADDRESS = {Providence, RI},
      YEAR = {1992},
   MRCLASS = {05E05 (15A18 60E99 62H10)},
  MRNUMBER = {1199126 (93j:05164)},
}

\bib{HeroWillenbring}{article}{
   author={Hero, Michael W.},
   author={Willenbring, Jeb F.},
   title={Stable Hilbert series as related to the measurement of quantum entanglement},
   JOURNAL = {Discrete Math.},
  FJOURNAL = {Discrete Mathematics},
VOLUME = {309},
      YEAR = {2009},
    NUMBER = {23-24},
     PAGES = {6508--6514},
      ISSN = {0012-365X},
     CODEN = {DSMHA4},
   MRCLASS = {13D40 (13A50 81P40)},
  MRNUMBER = {2558615 (2011a:13033)},
MRREVIEWER = {Frank D. Grosshans},
       DOI = {10.1016/j.disc.2009.06.021},
	REVIEW={\MR{2558615 (2011a:13033)}}
}

\bib{HWW}{article}{
    AUTHOR = {Hero, Michael W.},
	Author = {Willenbring, Jeb F.},
	Author = { Williams, Lauren Kelly},
     TITLE = {The measurement of quantum entanglement and enumeration of
              graph coverings},
 BOOKTITLE = {Representation theory and mathematical physics},
    SERIES = {Contemp. Math.},
    VOLUME = {557},
     PAGES = {169--181},
 PUBLISHER = {Amer. Math. Soc.},
   ADDRESS = {Providence, RI},
      YEAR = {2011},
   MRCLASS = {22E47 (05A19 05C30 81P40)},
  MRNUMBER = {2848925 (2012i:22021)},
       DOI = {10.1090/conm/557/11031},
       URL = {http://dx.doi.org/10.1090/conm/557/11031},
}

\bib{Hilbert}{article}{
    AUTHOR = {Hilbert, David},
     TITLE = {Ueber die {T}heorie der algebraischen {F}ormen},
   JOURNAL = {Math. Ann.},
  FJOURNAL = {Mathematische Annalen},
    VOLUME = {36},
      YEAR = {1890},
    NUMBER = {4},
     PAGES = {473--534},
      ISSN = {0025-5831},
     CODEN = {MAANA},
   MRCLASS = {Contributed Item},
  MRNUMBER = {1510634},
}











\end{biblist}
\end{bibdiv}
\end{document}